\documentclass[11pt,reqno]{amsart}

\usepackage{amssymb,dirtytalk,comment,tikz-cd,tikz,color,cite,enumerate, xcolor}
\usepackage[unicode=true]
 {hyperref}
\hypersetup{
colorlinks=true,
urlcolor=black,
citecolor=blue,
linkcolor=blue,
}

\newtheorem{theorem}{Theorem}[section]

\newtheorem{prop}[theorem]{Proposition}

\theoremstyle{definition}
\newtheorem{definition}[theorem]{Definition}
\newtheorem{example}[theorem]{Example}
\newtheorem*{question*}{Question}

\theoremstyle{remark}
\newtheorem{remark}[theorem]{Remark}
\newtheorem*{remark*}{Remark}

\numberwithin{equation}{section}

\DeclareMathOperator{\Aut}{Aut}
\DeclareMathOperator{\Ker}{Ker}

\usepackage{fancyhdr}
\begin{document}

\title{On Cardinalities whose Arithmetical Properties Determine the Structure of Solutions of the Yang--Baxter Equation}

\author{Maria Ferrara}
\address{Maria Ferrara --- Dipartimento di Matematica e Fisica, Università degli Studi della Campania ‘‘Luigi Vanvitelli’’}
\email{maria.ferrara1@unicampania.it}

\author{Marco Trombetti}
\address{Marco Trombetti --- Dipartimento di Matematica e Applicazioni ‘‘Renato Caccioppoli’’, Università degli Studi di Napoli Federico II}
\email{marco.trombetti@unina.it}

\author{Cindy (Sin Yi) Tsang}
\address{Cindy (Sin Yi) Tsang --- Department of Mathematics, Ochanomizu University, 2-1-1 Otsuka, Bunkyo-ku, Tokyo, Japan}
\email{tsang.sin.yi@ocha.ac.jp}
\urladdr{http://sites.google.com/site/cindysinyitsang/}



\begin{abstract} 
The aim of this paper is to provide purely arithmetical characterisations of those natural numbers $n$ for which every non-degenerate set-theoretic solution of cardinality $n$ of the Yang--Baxter equation arising from a skew brace (sb-solution for short) satisfies some relevant properties, such as being a flip or being involutive. For example, it turns out that every sb-solution of cardinality $n$ has finite multipermutation level if and only if its prime factorisation $n= p_1^{\alpha_1} \ldots p_t^{\alpha_t}$ is cube-free, namely $\alpha_i\leq 2$ for every $i$, and $p_i$ does not divide $p_j^{\alpha_j}-1$ for ~\hbox{$i\neq j$}. Two novel constructions of skew braces will play a central role in our proofs.

We shall also introduce the notion of supersoluble solution and show how this concept is related to that of supersoluble skew brace. In doing so, we have spotted an irreparable mistake in the proof of Theorem C [Ballester-Bolinches et al., Adv. Math. 455 (2024)], which characterizes soluble solutions in terms of soluble skew braces. 
\end{abstract}

\maketitle

\section{Introduction}

\noindent The Yang--Baxter equation (YBE, for short) is a consistency equation which was independently obtained by the physicists Yang \cite{Yang} and Baxter \cite{Baxter} in the field of quantum statistical mechanics. It has many relevant interpretations in the realm of mathematical physics, and besides that, it plays a key role in the foundation of quantum groups. Moreover, it provides a multidisciplinary approach for a wide variety of areas such as Hopf algebras, knot theory, and braid theory, among others. 

\smallskip

In this work, we focus on {\it non-degenerate set-theoretic solutions} (solutions, for short) of the YBE, that is, on the pairs $(X,r)$ where $X$ is a set and
\[r: (x,y)\in X\times X\longmapsto (\lambda_x(y),\rho_y(x))\in X\times X\]
is a bijective map for which the equality $r_{12}r_{23}r_{12}=r_{23}r_{12}r_{23}$ holds and the component maps $\lambda_x,\, \rho_x$ are bijective for every $x\in X$ --- here $r_{12}=r\times\operatorname{id}_X$ and $r_{23}=\operatorname{id}_X\times\, r$. Recall that $(X,r)$ is said to be {\it involutive} if $r^2 = \mathrm{id}_{X\times X}$. For every set $X$, there always exists the involutive solution $(X,r)$, where $r$ is defined by $r(x,y)=(y,x)$ --- this is referred to as the {\it flip} solution.

\begin{remark*}
The term ‘‘trivial solution’’ is usually reserved for flip solutions of cardinality at least $2$, while the only solution with one element is said to be a {\it one-element} solution. To simplify our terminology, we use the term “flip solutions” to encompass both trivial and one-element solutions.
\end{remark*}

In this work, we are interested in the solutions $(X,r)$ with a finite underlying set $X$. However, as observed in \cite{numbersolution}, the number of solutions grows very fast in terms of the cardinality. For example, as shown in \cite[Theorems 1.3 and 1.4]{numbersolution}, up to isomorphism, there are:
\begin{enumerate}[$\bullet$]
\item $321,931$ involutive solutions of cardinality $9$;
\item $4,895,272$ involutive solutions of cardinality $10$;
\item $422,449,480$ non-involutive solutions of cardinality $8$.
\end{enumerate}
It is therefore clear that, in order to classify all (finite) solutions of the YBE (this problem is still very far from being solved at the moment), additional restrictions must be imposed. The algebraic structure of skew braces is one of the main tools used to achieve this.

\smallskip

A {\it skew (left) brace} is a set $A$ endowed with two group structures $(A,+)$ and $(A,\circ)$ satisfying the {\it skew left distributivity}: $$a\circ(b+c)=a\circ b-a+a\circ c,$$ for all $a,b,c\in A$. For any group-theoretic property $\mathfrak X$, a skew brace is said to be {\it of $\mathfrak X$-type} if its additive group has property $\mathfrak X$. Thus, skew brace is a generalisation of {\it brace} as introduced by~Rump (see \cite{GV} and \cite{Rump0}), which (in our terminology) is just a skew brace of abelian type. 

\smallskip

In any skew brace $(A,+,\circ)$, it is easy to see that the identity $0$ of $(A,+)$ coincides with that of $(A,\circ)$. Also $(A,\circ)$ acts on $(A,+)$ via the {\it $\lambda$-map}: for every $a\in A$, the map $$\lambda_a:b\in A\longmapsto \lambda_a(b)=-a+a\circ b\in A$$ is an automorphism of $(A,+)$, and the map $$\lambda: a\in (A,\circ)\longmapsto\lambda_a \in \operatorname{Aut}(A,+)$$ is a group homomorphism. The ``distance" between the operations $+$ and $\circ$ is measured by the so-called {\it star product}: $$a\ast b=-a+a\circ b-b = \lambda_a(b) -b,$$ for all $a,b\in A$. In fact $a\circ b = a+b$ if and only if $a\ast b=0$ for all $a,b\in A$, in which case~$A$ is said to be {\it trivial}. Recall also that $A$ is said to be {\it almost trivial} if $a\circ b=b+a$ for all~\hbox{$a,b\in A$.} 

\smallskip

For every (finite) skew brace $(A,+,\circ)$, one can naturally associate to it a (finite) solution $(A,r_A)$ of the YBE defined by
\begin{equation}\label{rA} r_A: (a,b) \in A\times A \longmapsto (\lambda_a(b),\lambda_a(b)^{-1}\circ a \circ b)\in A\times A,\end{equation}
which is involutive if and only if $A$ is a brace (see \cite{GV}, \cite{Rump0}, and also \cite{ChildsYBE}). We shall refer to these solutions that arise from a skew brace as {\it skew-brace-solutions} ({\it sb-solution} for short). Conversely, for every (finite) solution $(X,r)$ of the~YBE, one can associate to it the (not necessarily finite) {\it structure group} of $(X,r)$, defined by the presentation
\[G(X,r)=\langle x\in X\,:\, x \circ y=u\circ v\textnormal{ for }r(x,y)=(u,v)\rangle,\]
on which one can define a group operation $+$ such that $(G(X,r),+,\circ)$ is a skew brace satisfying a certain universal property (see \cite[Theorem 3.5]{SV}).

\smallskip

In contrast to solutions of the YBE, the number of skew braces does not grow as rapidly in terms of the order. For example, by \cite[Tables 5.1 and~5.3]{GV}, up to isomorphism, there are only:
\begin{enumerate}[$\bullet$]
\item $4$ braces of order $9$;
\item $2$ braces of order $10$;
\item $20$ skew braces of order $8$ that are not braces.
\end{enumerate}
It is therefore clear that, when trying to classify all (finite) solutions of the~YBE, we may adjust the level of difficulty by restricting to sb-solutions. For example, this was the approach of \cite{ballester}, where soluble sb-solutions have been characterised in terms of solubility of the associated skew brace, although, as shown in Remark \ref{irreparable}, their approach is not really satisfactory.

\smallskip

As shown in \cite{relation}, the cardinality of a solution may give many information about its properties --- the main theorem states that any indecomposable involutive solution of the YBE of square-free cardinality is a {\it multipermutation} solution. We refer the reader to Section \ref{sec:Prelim} for the terminology, but what is relevant here is that multipermutation solutions have a controllable level of complexity, so knowing that solutions of a certain cardinality are always multipermutation is really a good thing.  

\smallskip

In this paper, we have obtained purely arithmetical characterisations of the natural numbers $n$ for which {\it every} sb-solution of cardinality $n$ satisfies some relevant properties --- we shall consider the properties of being a flip, involutive, multipermutation, and supersoluble. The notion of {\it supersoluble} solution is introduced for the first time in this paper (see Section \ref{sec:Prelim} for the definition) and was inspired by the concept of soluble solution given in \cite{ballester}. 


\smallskip

\noindent{\bf Theorem A}\quad Let $n$ be a natural number, and let $p_1^{\alpha_1}\ldots p_t^{\alpha_t}$ be its prime factorisation. Then the following are equivalent:
\begin{enumerate}[$(1)$]
    \item Every sb-solution of cardinality $n$ is a flip solution.
    \item $\alpha_i=1$ for every $i$, and $p_i$ does not divide $p_j-1$ for $i\neq j$.
\end{enumerate}


\smallskip

\noindent{\bf Theorem B}\quad Let $n$ be a natural number, and let $p_1^{\alpha_1}\ldots p_t^{\alpha_t}$ be its prime factorisation. Then the following are equivalent:
\begin{enumerate}[$(1)$]
    \item Every sb-solution of cardinality $n$ is multipermutation.
    \item Every sb-solution of cardinality $n$ is involutive.
    \item $\alpha_i\leq 2$ for every $i$, and $p_i$ does not divide $p_j^{\alpha_j}-1$ for $i\neq j$.
\end{enumerate}

\eject

\noindent{\bf Theorem C}\quad Let $n$ be a natural number, and let $p_1^{\alpha_1}\ldots p_t^{\alpha_t}$ be its prime factorisation. Suppose that $n$ satisfies the following conditions: 
    \begin{enumerate}[$\bullet$]
    \item $\alpha_i\leq 2$ for every $i$;
    \item If $\alpha_j=2$, then $p_i$ does not divide $p_j^2-1$ for $i\neq j$; 
    \item If $4$ divides $n$, then $p_i\equiv1\pmod{4}$ for every $i$ with $\alpha_i=2$. 
\end{enumerate}
\noindent Then every sb-solution of cardinality~$n$ is supersoluble.
\smallskip

The difference in the statements of Theorems A, B and Theorem C comes from the following fact. The most relevant result of \cite{ballester} is a characterisation of soluble solutions in terms of soluble skew braces. However, as we shall soon see in~Re\-mark \ref{irreparable}, the proof of this result contains an irreparable gap and it is actually very unlikely that such a characterisation can be achieved without the addition of very strong non-solution-theoretic conditions.

Our three main theorems will be obtained as corollaries of more general results on arithmetical characterisations of the natural numbers $n$ for which every skew brace of order $n$ satisfies a certain algebraic property (see~The\-o\-rems~\ref{thm:cyclic}, \ref{thm:abelian}, \ref{thm:nilpotent}, and~\ref{super}). In the course of the proof, we shall introduce two new constructions of skew braces that are of independent interests (see~The\-o\-rems~\ref{thm:first method} and \ref{thm:second method}). Theorem B should also be seen in connection with the problem of establishing a rigorous framework to prove that “almost all” solutions are multipermutation (see \cite[Pro\-blem~5.11]{leandro}). 

\section{Preliminaries}\label{sec:Prelim}

\noindent The aim of this section is to recall some basic results and definitions that are needed to prove our main theorems. 

\smallskip

Let $(A,+,\circ)$ be a skew brace. A subset $X$ of $A$ is said to be:
\begin{enumerate}[(1)]
    \item a {\it sub-skew brace} if it is a subgroup of both $(A,+)$ and $(A,\circ)$;
    \item a {\it left-ideal} if it is a subgroup of $(A,+)$ and $\lambda_a(X)=X$ for all $a\in A$; a left-ideal is automatically a sub-skew brace;
    \item an {\it ideal} if it is a left-ideal that is normal in both $(A,+)$ and $(A,\circ)$; in this case $A/I$ is a skew brace with induced operations.
\end{enumerate}
For example, the kernel $\Ker(\lambda)$ of the lambda map $\lambda$ is always a sub-skew brace of $A$, and the characteristic subgroups of $(A,+)$ are all left-ideals of $A$. There are two relevant ideals that often pop up in the study of solubility and nilpotency of skew braces: the {\it socle}, defined as 
\[\operatorname{Soc}(A)\!=\!Z(A,+)\cap\operatorname{Ker}(\lambda),\]
and the {\it annihilator}, defined as 
\[\operatorname{Ann}(A)=Z(A,\circ)\cap\operatorname{Soc}(A),\]
where~\hbox{$Z(A,+)$} and $Z(A,\circ)$ denote, respectively, the centres of $(A,+)$ and $(A,\circ)$. Note that the annihilator was first introduced in \cite{stefanelli} in the context of ideal extension of skew braces and later studied in \cite{bonatto}. There is also the {\it derived ideal} of $A$, defined as $A^2 = A\ast A$, which plays an important role in the study of skew braces. Here, as usual, for any subsets $X$ and $Y$ of $A$, we put $X\ast Y$ to be the subgroup of the additive group $(A,+)$ generated by the elements~$x\ast y$, where $x\in X$ and $y\in Y$.


\smallskip

Our main results deal with many algebraic properties of skew braces, and we now briefly explore them for the reader's convenience.

\smallskip

A skew brace $(A,+,\circ)$ is said to be a {\it bi-skew brace} (or a {\it symmetric} skew brace according to some authors \cite{symmetric}) if $(A,\circ,+)$ is also a skew brace. This concept was first introduced by Childs in \cite{Childs}, and his main focus was the connection between skew braces and Hopf--Galois theory. More specifically, by the Greither--Pareigis correspondence \cite{GP} and Byott's translation \cite{Byott}, a finite skew brace $(A,+,\circ)$ gives rise to a Hopf--Galois structure of type $(A,+)$ on any Galois extension with Galois group isomorphic to $(A,\circ)$ (see \cite{Childs book} and also \cite{ST}). The consideration of bi-skew braces allows one to switch the type of the Hopf--Galois structure and the Galois group of the extension.

\smallskip

A skew brace $(A,+,\circ)$ is said to be {\it two-sided} if in addition to the skew left distributivity, the {\it skew right distributivity} also holds, that is, if $$(a+b)\circ c=a\circ c-c+b\circ c,$$ for all $a,b,c\in A$. Clearly, every skew brace having an abelian multiplicative group is two-sided. It is also known by \cite{Rump0} that two-sided braces are exactly the braces that arise from radical rings.

\smallskip

Bi-skew braces and two-sided skew braces are much easier to handle. For example, it was conjectured by Byott \cite{soluble} that a finite skew brace whose additive group is soluble cannot have an insoluble multiplicative group. Some  significant progress was made in \cite{Byott transitive}, but this conjecture is still open. Nevertheless, it is known to be true when restricted to bi-skew braces and two-sided skew braces (see \cite[Theorem 3.11]{ST2} and \cite[Theorem 4.3]{two-sided}). In terms of the lambda map, a skew brace $(A,+,\circ)$ is a bi-skew brace if and only if
\begin{equation}\label{biskewbrace}
\lambda_{ab}=\lambda_b\lambda_a\quad\textnormal{and}\quad \lambda_{\lambda_a(b)} = \lambda_a\lambda_b\lambda_a^{-1}
\end{equation}
for all $a,b\in A$ (see \cite[Theorem 3.1]{Caranti}), and by its proof $\lambda_a\in \Aut(A,\circ)$ for all $a\in A$ in this case. In terms of the star product, while we only have 
\[ a * (b+c) = a*b + b + 
a*c -b\]
in an arbitrary skew brace $(A,+,\circ)$, we also have the identity
\[ (a+b)*c = -b + a*c + b + b*c\]
in a two-sided skew brace $(A,+,\circ)$. These nice properties make calculations a lot simpler in many occasions.

\smallskip

A skew brace $(A,+,\circ)$ is said to be {\it weakly trivial} if $A^2 \cap A_{\textnormal{op}}^2 = \{0\}$. Here $A_{\textnormal{op}}=(A,+_{\textnormal{op}},\circ)$, where $+_{\textnormal{op}}$ is defined by $a+_{\textnormal{op}} b = b+a$ for all $a,b\in A$, denotes the {\it opposite skew brace} of $A$ as defined in \cite{opposite}. This concept first appeared in \cite{Trap} as a tool to study two-sided skew braces, and it was shown in \cite[Corollary 4.4]{Trap} that every two-sided skew brace is an extension of a weakly trivial skew
brace by a two-sided brace.

\smallskip

A skew brace $(A,+,\circ)$ is said to be {\it $\lambda$-homomorphic} if its lambda map $\lambda$ is not only a homomorphism on $(A,\circ)$ but also on $(A,+)$. This definition is due to \cite{lambdahom}, where it was applied to construct skew braces of free or free-abelian type. Clearly, the derived ideal of a $\lambda$-homomorphic skew brace is contained in $\Ker(\lambda)$, so a $\lambda$-homomorphic skew brace is {\it meta-trivial} in the sense that its derived ideal is trivial as a skew brace. Using \eqref{biskewbrace}, it is also easy to check that a $\lambda$-homomorphic skew brace with abelian image $\mathrm{Im}(\lambda)$ is a bi-skew brace (also see \cite[Corollary 4.6]{symmetric}).

\smallskip

A skew brace $(A,+,\circ)$ is said to be {\it one-generator} if there exists $a\in A$ such that the smallest sub-skew brace containing $a$ is $A$. In case of braces, this concept has an unexpected relationship with indecomposable involutive solutions to the YBE (see \cite{Rump01}). As shown in \cite{one-generator}, among the one-generator braces $A$ for which $A*A^2 = \{0\}=A^2*A$, there is a universal brace with additive group $\mathbb{Z}\times\mathbb{Z}$ that admits all such braces as an epimorphic image.

\subsection{Multipermutation solutions and nilpotency of skew braces}

\hfill\smallskip

\noindent Let $(X,r)$ be a solution of the YBE and write $r(x,y)=(\lambda_x(y),\rho_y(x))$ for all $x,y\in X$. Define an equivalence relation $\sim$ on $X$ by putting 
\[ x\sim y\,\ \iff \,\ (\lambda_x=\lambda_y\,\ \textnormal{and} \,\ \rho_x=\rho_y)\]
for all $x,y\in X$. Then $\operatorname{Ret}(X,r)=(\overline{X},\overline r)$, where $\overline{X} = X/{\sim}$ and
\[ \overline{r} : ([x],[y])\in \overline{X}\times\overline{X}\longmapsto ([\lambda_x(y)],[\rho_y(x)])\in \overline{X}\times\overline{X},\]
is also a solution of the YBE, called the {\it retraction} of $(X,r)$. By recursion, we can then define 
\[\operatorname{Ret}^0(X,r) = (X,r),\quad \operatorname{Ret}^{m+1}(X,r)=\operatorname{Ret}(\operatorname{Ret}^m(X,r))\]
for all $m\geq0$. We say that $(X,r)$ is {\it multipermutation} if the underlying set of $\operatorname{Ret}^m(X,r)$ becomes singleton for some $m$. 

\smallskip

Now, let $(A,+,\circ)$ be a skew brace. By taking socle or annihilator recursively, we can define the {\it socle series} by putting 
\[\operatorname{Soc}_{0}(A)=\{0\},\quad \operatorname{Soc}_{m+1}(A)/\mathrm{Soc}_m(A)=\operatorname{Soc}(A/\mathrm{Soc}_m(A))\]
for all $m\geq 0$, and the {\it annihilator series} by putting 
\[\operatorname{Ann}_{0}(A)=\{0\},\quad \operatorname{Ann}_{m+1}(A)/\mathrm{Ann}_m(A)=\operatorname{Ann}(A/\mathrm{Ann}_m(A))\]
for all $m\geq 0$. They are analogs of the upper central series. Following \cite{nilpotent}, we say that $A$ has {\it finite multipermutation level} if $\operatorname{Soc}_m(A)=A$ for some $m$. Similarly, we say that $A$ is {\it annihilator nilpotent} (or {\it centrally nilpotent}) if $\operatorname{Ann}_m(A)=A$ for some $m$. Clearly, if $A$ is annihilator nilpotent, then $A$ has finite multipermutation level.

\smallskip

The following result characterises multipermutation sb-solutions in terms of the multipermutation level of their associated skew braces (see \cite[Proposition 5.3]{Jordan}).

\begin{theorem}\label{thm:fml}
Let $(X,r)$ be an sb-solution and let $A$ be its associated skew brace. Then $(X,r)$ is multipermutation if and only if $A$ has finite multipermutation level. In fact, the smallest $m$ for which $|\mathrm{Ret}^m(X,r)|=1$ coincides with that for which $\mathrm{Soc}_m(A)=A$.
\end{theorem}

Let $(A,+,\circ)$ be a skew brace. By taking the star product recursively, we can define the {\it left series} by putting
\[ A^1 = A,\quad A^{m+1} = A* A^m\]
for all $m\geq 1$, and the {\it right series} by putting
\[ A^{(1)}=A,\quad A^{(m+1)} = A^{(m)} *A \]
for all $m\geq 1$. They are analogs of the lower central series. We say that $A$ is {\it left-nilpotent} if $A^m=\{0\}$ for some $m$, and similarly that $A$ is {\it right-nilpotent} if $A^{(m)}=\{0\}$ for some $m$.

\smallskip

The property of having finite multipermutation level can be more easily detected in case the skew brace is of nilpotent type. Indeed, for any skew brace $A$ of nilpotent type, by \cite[Lemma 2.16]{nilpotent} we know that
\[ \mbox{$A$ has finite multipermutation level}\,\ \iff \,\ 
\mbox{$A$ is right-nilpotent},\]
and similarly, by \cite[Corollary 2.15]{JVV} we know that
\[ \mbox{$A$ is annihilator nilpotent}\,\ \iff \,\ \mbox{$A$ is both left- and right-nilpotent}.\] 
Some further properties of annihilator nilpotency are described in \cite{centralnilp}, \cite{bonatto}, \cite{nilpotent}, \cite{colazzo}, and \cite{Tr23}. For example, a finite skew brace $(A,+,\circ)$ is annihilator nilpotent only when $(A,+)$ and $(A,\circ)$ are both nilpotent (see \cite[Corollary~2.11]{bonatto}), in which case the~Sy\-low subgroups of $(A,+)$ are all ideals and~$A$ is a direct product of them (this is a well-known fact, which we explicitly state below since we need it in the proofs of our main theorems). It follows that a finite skew brace is annihilator nilpotent if and only if the additive~Sy\-low subgroups are all ideals that are annihilator nilpotent as skew braces (also see \cite[Theorem 4.13]{centralnilp} for a ``local" version of this).

\begin{prop}\label{prop:nilpotent}
Let $(A,+,\circ)$ be a finite skew brace whose additive and multiplicative groups are nilpotent. Then, for each prime $p$, the Sylow~\hbox{$p$-sub}\-group of $(A,+)$ is an ideal, and $A$ is the direct product of these ideals.
\end{prop}

\subsection{Supersoluble solutions and supersolubility of skew braces}\label{supersolublesec}

\hfill\smallskip

\noindent Recall that a finite skew brace $(A,+,\circ)$ is said to be {\it supersoluble} if it has a finite series of ideals 
\[\{0\}=I_0\subseteq I_1\subseteq\ldots\subseteq I_m=A\]
such that $I_{i+1}/I_i$ has prime order for every $i=0,\ldots,m-1$. 
This concept was introduced in \cite{supersoluble} for the first time, where a lot of nice properties were shown. For example, every finite skew brace of square-free order (more generally, every finite skew brace all of whose additive and multiplicative Sylow subgroups are cyclic) is supersoluble (see \cite[Theorem 3.8]{supersoluble}). Supersolubility also offers a large setting in which skew brace exhibits desirable behaviors. For example, although the sum of two annihilator nilpotent ideals need not be annihilator nilpotent in general (see \cite[Example B]{centralnilp}), such is the case in the context of supersoluble skew braces (see \cite[Corollary 3.37]{supersoluble}). 

\smallskip

Every skew brace of prime order is a trivial brace, so every finite supersoluble skew brace is {\it soluble} in the sense of \cite[Definition 18]{ballester}. Here, it follows from \cite[Proposition 20]{ballester} that a skew brace $(A,+,\circ)$ is soluble if and only it has a finite series of ideals
\[\{0\}=I_0\subseteq I_1\subseteq\ldots\subseteq I_m=A\]
such that $I_{i+1}/I_i$ is a trivial brace for every $i=0,\ldots,m-1$. The notion of soluble solution of the YBE was also introduced in \cite[Definition 1]{ballester}, and the following definition should be seen in comparison with \cite{ballester}. In \cite[The\-o\-rems~C and D]{ballester}, the relationship between soluble solutions and soluble skew braces was discussed. 

\begin{remark}\label{irreparable}
There is an irreparable mistake in the proof of the backward implication of \cite[Theorem C]{ballester} --- in the notation there the step
\[ \Ker(g)\cap B = 0\,\ \implies \,\ \{f_k(0)\} = \overline{f}_k(\Ker(g))\cap Y_k\]
in line 1 on p.19 is not valid because intersection is not preserved under mappings in general. Therefore, in our Definition \ref{supersolutions}, we impose other necessary conditions that were not present in \cite[Definition 1]{ballester}, and our The\-o\-rem~\ref{thm:supersoluble solution} only gives a sufficient condition for a finite sb-solution to be supersoluble. Let us also mention that the proof of \cite[Theorem D]{ballester} on p.19 seems to contain a gap as well --- in the notation there the required morphism $f_n$ was not specified, and even if one takes $f_n$ to be the canonical map
\[f_n : x\in X\longmapsto \iota(x)\in G(X,r),\]
which seems to be the natural choice based on the definition of $f_1,\dots,f_{n-1}$, one cannot show that $X_n$ is an equivalence class under $\sim_{f_n}$ (see below for the definition) because $\iota$ is not injective in general.
\end{remark}

Let $(X,r)$ and $(Y,s)$ be any solutions of the YBE. Recall that a homomorphism $f: (X,r)\rightarrow (Y,s)$ is a map such that the diagram
\[\begin{tikzcd}[column sep=1.5cm,row sep=1.25cm]
X \times X \arrow{r}{r} \arrow{d}[left]{f\times f}& X\times X\arrow{d}{f\times f}\\
Y\times Y\arrow{r}[below]{s} &  Y\times Y
\end{tikzcd}\]
commutes. Write $r(x,y) = (\lambda_x(y),\rho_y(x))$ for $x,y\in X$ as usual. Note that if~$Z$ is a subset of $X$ such that $\lambda_x(Z) = Z$ and $\rho_x(Z)=Z$ for all $x\in X$, then clearly $r(Z\times Z)=Z\times Z$
and so $r$ induces a solution $(Z,r|_{Z\times Z})$ via restriction. In this case 
$s(f(Z)\times f(Z)) = f(Z)\times f(Z)$ by the commutativity of the above diagram, and so $s$ induces a solution $(f(Z),s|_{f(Z)\times f(Z)})$ via restriction. Now, we can define an equivalence relation $\sim_f$ on $X$ by putting
\[ x\sim_f y \,\ \iff \,\ f(x)=f(y)\]
for all $x,y\in X$. Using this notation, we give the following definition.

\begin{definition}\label{supersolutions}
A finite solution $(X,r)$ of the YBE is said to be {\it supersoluble at $x_0$} if there exists a series of subsets 
\[\{x_0\}=X_0\subseteq X_1\subseteq\ldots\subseteq X_m=X\]
such that there are solutions $(Y_i,s_i)$ and morphisms $f_i:(X,r)\rightarrow(Y_i,s_i)$ of solutions for $i=0,\dots,m-1$ satisfying all of the following conditions:
\begin{enumerate}[$(1)$]
\item $X_i$ is an equivalence class under $\sim_{f_i}$;
\item the equivalence classes under $\sim_{f_i}$ all have size $|X_i|$;
\item $X_{i+1}$ is the union of a collection of equivalence classes under $\sim_{f_i}$;
\item $\lambda_x(X_{i+1})=X_{i+1}$ and $\rho_x(X_{i+1})=X_{i+1}$ for all $x\in X$; 
\item $(f_{i}(X_{i+1}),s_i|_{f_{i}(X_{i+1})\times f_{i}(X_{i+1})})$ is a trivial solution of prime cardinality.
\end{enumerate}
An sb-solution $(X,r)$ is said to be {\it supersoluble} if it is supersoluble at $0$.
\end{definition}


\begin{theorem}\label{thm:supersoluble solution}
Let $(X,r)$ be a finite sb-solution and let $A$ be its associated skew brace. If $A$ is supersoluble, then $(X,r)$ is supersoluble.
\end{theorem}
\begin{proof}
Suppose that $A$ is supersoluble and let 
\[\{0\}=I_0\subseteq I_1\subseteq\ldots\subseteq I_m=A\]
be a series of ideals of $A$ such that the consecutive factors have prime order. For each $i=0,\dots,m-1$, clearly the canonical epimorphism
\[f_i:a\in (A,r_A)\longmapsto a+I_i\in  (A/I_i,r_{A/I_i})\]
of skew braces is also an epimorphism of solutions. The equivalence classes of $\sim_{f_i}$ are precisely the cosets of $I_i=\operatorname{Ker}(f_i)$, so conditions (1), (2), and (3) are clear. As for condition (4), for any $a\in A$, since $I_{i+1}$ is an ideal of $A$ we plainly have $\lambda_a(I_{i+1})=I_{i+1}$ and $\rho_a(I_{i+1})=I_{i+1}$.
 Note that $I_{i+1}/I_{i}$ is a trivial brace because it has prime order. It then follows that
\begin{align*} (f_i(I_{i+1}), r_{A/I_i}|_{f_i(I_{i+1})\times f_i(I_{i+1})})& = (I_{i+1}/I_i, r_{A/I_i}|_{I_{i+1}/I_i\times I_{i+1}/I_i}) \\ & = (I_{i+1}/I_i,r_{I_{i+1}/I_i}) \end{align*}
is a trivial solution of prime cardinality, so condition (5) also holds.
\end{proof}

\section{Constructions of skew braces and examples}

\noindent Before we prove our main results, we shall first introduce two new methods to construct skew braces from groups using semi-direct products, and then we shall apply them to give some related examples. 

\smallskip

Let us first explain the ideas behind the constructions. Let $(A,+,\circ)$ be a skew brace and suppose that
\[ (A,+) = (B,+)\rtimes (C,+),\quad (A,\circ ) = (B,\circ)\rtimes (C,\circ)\]
for some ideal $B$ and sub-skew brace $C$. For simplicity, we assume that $B$ and $C$ are trivial skew braces, so that we can regard them as groups. Note that conjugation by $(C,+)$ induces a homomorphism
\[ \phi : c\in C\longmapsto\phi_c :=(b\longmapsto c + b-c)\in \Aut(B),\]
and similarly conjugation by $(C,\circ)$ induces a homomorphism
\[ \psi : c\in C \longmapsto \psi_c := (b\longmapsto c \circ b \circ{c}^{-1}) \in \Aut(B).\]
Since $B$ is an ideal, the lambda map also induces a homomorphism
\[
\gamma : c\in C \longmapsto\gamma_c:=(b\longmapsto \lambda_c(b))\in \Aut(B).
\]
In what follows, let $b,b_1,b_2\in B$ and $c,c_1,c_2\in C$. Clearly, we have
\[ (b_1+c_1) + (b_2+c_2) = (b_1+\phi_{c_1}(b_2)) + (c_1+c_2).\] Now, the $\circ$-product of two elements of $(A,+)$ can be expressed in two different ways,  leading to two different constructions. 

\smallskip

First, assume $\gamma=\psi$. Then we can write
\[ b\circ c = c\circ \psi_{c}^{-1}(b) = c\circ \gamma_{c}^{-1}(b) = c+b =\phi_c(b) + c.\]
Since $B$ and $C$ are trivial skew braces, we have
\[
(b_1\circ c_1)\circ (b_2\circ c_2) = (b_1+ \psi_{c_1}(b_2))\circ (c_1+ c_2),\]
which we can write as
\[ (\phi_{c_1}(b_1)+c_1)\circ(\phi_{c_2}(b_2)+c_2) = \phi_{c_1+c_2}(b_1+\psi_{c_1}(b_2)) + (c_1+c_2).\]
The first construction~(see The\-o\-rem~\ref{thm:first method}) is based on this equality.

\smallskip

Next, without any assumption on $\gamma$ and $\psi$, we can write
\[ c\circ b = c+\gamma_c(b) = (\phi_c\gamma_c)(b) + c.\]
Again, since $B$ and $C$ are trivial skew braces, we have 
\[(c_1\circ b_1)\circ (c_2\circ b_2) = (c_1+ c_2) \circ (\psi_{c_2}^{-1}(b_1)+b_2),\]
which we can write as
\begin{align*}
&\big((\phi_{c_1}\gamma_{c_1})(b_1)+c_1\big)\circ \big((\phi_{c_2}\gamma_{c_2})(b_2)+c_2\big) \\
&\hspace{2cm}= (\phi_{c_1+c_2}\gamma_{c_1+c_2})(\psi_{c_2}^{-1}(b_1)+b_2) +(c_1+c_2).
\end{align*}
The second construction (see Theorem \ref{thm:second method}) is based on this observation. 

\smallskip

For both constructions we shall assume that $B$ is an abelian group. 

\begin{theorem}\label{thm:first method}
    Let $B=(B,+)$ be an abelian group and $C=(C,\cdot)$ a group. Given any homomorphisms $\phi,\psi : c\in C\longmapsto \phi_c,\psi_c\in\Aut(B)$, define
\[ (b_1,c_1) + (b_2,c_2) = (b_1 + \phi_{c_1}(b_2),c_1c_2)\]
and 
\[ (\phi_{c_1}(b_1),c_1) \circ (\phi_{c_2}(b_2),c_2) =  \big(\phi_{c_1c_2}\big(b_1+\psi_{c_1}(b_2)\big),c_1c_2\big)\]
for all $b_1,b_2\in B$ and $c_1,c_2\in C$. 

\smallskip

Then $(B\times C,+,\circ)$ is a skew brace if and only if the relation
\begin{equation}\label{eq:first method} 
\phi_c\psi_{c'}=\psi_{c'}\phi_{c}
\end{equation}
holds for all $c,c'\in C$. In this case, we have\textnormal:
\begin{enumerate}[$(a)$]
\item $(B\times C,+,\circ)$ is two-sided if and only if
\[ \phi_{c}\phi_{c'}=\phi_{c'}\phi_{c}\quad\textnormal{and}\quad (\phi_c\psi_c-\mathrm{id}_B)(\psi_{c'}-\mathrm{id}_B)=0\]
hold for all $c,c'\in C$.
\item $(B\times C,+,\circ)$ is a bi-skew brace if and only if
\[ \psi_c\psi_{c'}=\psi_{c'}\psi_c\quad\textnormal{and}\quad (\phi_c\psi_c-\mathrm{id}_B)(\phi_{c'}-\mathrm{id}_B)=0\]
hold for all $c,c'\in C$.
\item $(B\times C,+,\circ)$ is $\lambda$-homomorphic if and only if
\[ (\phi_c-\mathrm{id}_B)(\phi_{c'}-\mathrm{id}_B)=0\quad\textnormal{and}\quad (\phi_c-\mathrm{id}_B)(\psi_{c'}-\mathrm{id}_B)=0\]
hold for all $c,c'\in C$.
\end{enumerate}
\end{theorem}

\begin{proof} Clearly $(B\times C,+)$ is a group, while $(B\times C,\circ)$ is a group because $\circ$ is the operation on $B\times C$ induced via transport by the bijection
\[ (b,c)\in B\times C\longmapsto (\phi_{c}^{-1}(b),c)\in B\rtimes_\psi C.\]
In what follows, define 
\[ a_1= (\phi_{c_1}(b_1),c_1),\quad a_2=(\phi_{c_2}(b_2),c_2),\quad a_3= (\phi_{c_3}(b_3),c_3)\]
with $b_1,b_2,b_3\in B$ and $c_1,c_2,c_3\in C$. 

\smallskip

For the skew left distributivity, observe that
\begin{align*}
&\hspace{5mm} a_1\circ (a_2+a_3)\\
& = a_1\circ \big(\phi_{c_2}(b_2+\phi_{c_3}(b_3)),c_2c_3\big)\\
& = \big(\phi_{c_1c_2c_3}(b_1 + \psi_{c_1}(\phi_{c_3}^{-1}(b_2) + b_3)),c_1c_2c_3\big)
\end{align*}
on the one hand, and
\begin{align*}
&\hspace{5mm}(a_1\circ a_2)-a_1 + (a_1\circ a_3)\\
&= \big(\phi_{c_1c_2}(b_1+\psi_{c_1}(b_2)),c_1c_2\big)+(-b_1,c_1^{-1}) + \big(\phi_{c_1c_3}(b_1+\psi_{c_1}(b_3)),c_1c_3\big)\\
& = \big((\phi_{c_1c_2}\psi_{c_1})(b_2),c_1c_2c_1^{-1}\big) + \big(\phi_{c_1c_3}(b_1+\psi_{c_1}(b_3)),c_1c_3\big)\\
&= \big((\phi_{c_1c_2}\psi_{c_1})(b_2) + \phi_{c_1c_2c_3}(b_1+\psi_{c_1}(b_3)),c_1c_2c_3\big)\end{align*}
on the other. Note that the terms involving $b_1$ and $b_3$ are always equal. By comparing the terms with $b_2$, we then see that $(B\times C,+,\circ)$ is a skew brace if and only if the relation
\[ \phi_{c_3}\psi_{c_1}\phi_{c_3}^{-1}=\psi_{c_1}\]
always holds, which is as claimed.

\smallskip

Now, suppose that \eqref{eq:first method} holds.

\smallskip

\noindent \underline{Proof of (a)}

\smallskip

\noindent For the skew right distributivity, we have
\begin{align*}
&\hspace{5mm}(a_1+a_2)\circ a_3\\
& = \big(\phi_{c_1}(b_1+\phi_{c_2}(b_2)),c_1c_2\big)\circ a_3\\
& = \big(\phi_{c_1c_2c_3}(\phi_{c_2}^{-1}(b_1) + b_2 + \psi_{c_1c_2}(b_3)),c_1c_2c_3\big)
\end{align*}
on the one hand, and
\begin{align*}
&\hspace{5mm} (a_1\circ a_3)-a_3+(a_2\circ a_3)\\
& = \big(\phi_{c_1c_3}(b_1+\psi_{c_1}(b_3)),c_1c_3\big) + (-b_3,c_3^{-1})+\big(\phi_{c_2c_3}(b_2+\psi_{c_2}(b_3)),c_2c_3\big)\\
& = \big(\phi_{c_1c_3}
(b_1+\psi_{c_1}(b_3)-b_3),c_1\big) + \big(\phi_{c_2c_3}(b_2+\psi_{c_2}(b_3)),c_2c_3\big)\\
& = \big(\phi_{c_1c_3}(b_1+\psi_{c_1}(b_3)-b_3)+\phi_{c_1c_2c_3}(b_2+\psi_{c_2}(b_3)),c_1c_2c_3\big)
\end{align*}
on the other. Note that the terms involving $b_2$ are always equal. By comparing the terms with $b_1$ and $b_3$, respectively, we see that $(B\times C,+,\circ)$ is two-sided if and only if both 
\[\phi_{c_2c_3}\phi_{c_2}^{-1} =\phi_{c_3}\quad\textnormal{and}\quad \phi_{c_2c_3}\psi_{c_1c_2} = \phi_{c_3}(\psi_{c_1}-\mathrm{id}_B) + \phi_{c_2c_3}\psi_{c_2}\]
always hold. Using the former, the latter can be simplified to
\[ \phi_{c_2} (\psi_{c_1}-\mathrm{id}_B)\psi_{c_2} = \psi_{c_1}-\mathrm{id}_B,\]
and the claim (a) now follows from \eqref{eq:first method}.

\smallskip

\noindent \underline{Proof of (b)}

\smallskip

\noindent For the skew left distributivity with $+$ and $\circ$ reversed, we have
\begin{align*}
&\hspace{5mm} a_1 + (a_2\circ a_3)\\
&= a_1 + \big(\phi_{c_2c_3}(b_2+\psi_{c_2}(b_3)),c_2c_3\big)\\
& = \big(\phi_{c_1}(b_1+\phi_{c_2c_3}(b_2+\psi_{c_2}(b_3))),c_1c_2c_3\big)
\end{align*}
on the one hand, and
\begin{align*}
&\hspace{5mm}(a_1+a_2)\circ a_1^{-1}\circ (a_1+a_3)\\
& = \big(\phi_{c_1}(b_1+\phi_{c_2}(b_2)),c_1c_2\big)\circ \big((\phi_{c_1}^{-1}\psi_{c_1}^{-1})(-b_1),c_1^{-1}\big)\circ (a_1+a_3)\\
& = \big(\phi_{c_1c_2c_1^{-1}}\big(\phi_{c_2}^{-1}(b_1)+b_2 + (\psi_{c_1c_2}\psi_{c_1}^{-1})(-b_1)\big),c_1c_2c_1^{-1}\big)\circ (a_1+a_3)\\
&= \big(\phi_{c_1c_2c_3}\big(\phi_{c_2}^{-1}(b_1)+b_2 - \psi_{c_1c_2c_1^{-1}}(b_1) + \psi_{c_1c_2c_1^{-1}}(\phi_{c_3}^{-1}(b_1)+b_3)\big),c_1c_2c_3\big)
\end{align*}
on the other. Note that the terms involving $b_2$ are always equal. By comparing the terms with $b_3$ and $b_1$, respectively, we see that $(B\times C,+,\circ)$ is a bi-skew brace if and only if both 
\[ \psi_{c_2} = \psi_{c_1c_2c_1^{-1}} \quad\textnormal{and}\quad \mathrm{id}_B = \phi_{c_2c_3}\big(\phi_{c_2}^{-1}+\psi_{c_1c_2 c_1^{-1}}( \phi_{c_3}^{-1}-\mathrm{id}_B)\big) \]
always hold. Using the former, the latter can be simplified to
\[ \phi_{c_3}^{-1}-\mathrm{id}_B = \psi_{c_2}(\phi_{c_3}^{-1}-\mathrm{id}_B)\phi_{c_2},\]
and the claim (b) now follows from \eqref{eq:first method}.

\smallskip

\noindent \underline{Proof of (c)}

\smallskip

\noindent For any $b,x\in B$ and $c,y\in C$, we have
\begin{align*} \lambda_{(\phi_c(b),c)}(\phi_y(x),y) &= 
(-b,c^{-1}) + \big(\phi_{cy}(b+\psi_c(x)),cy\big)\\
& = \big((\phi_y-\mathrm{id}_B)(b) + (\phi_y\psi_c)(x),y\big).
\end{align*}
Using this, we get that
\[ \lambda_{a_1+a_2}(a_3) = \big((\phi_{c_3}-\mathrm{id}_B)(\phi_{c_2}^{-1}(b_1)+b_2)+(\phi_{c_3}\psi_{c_1c_2})(b_3),c_3\big)\]
on the one hand, and
\[ \lambda_{a_1\circ a_2}(a_3) = \big((\phi_{c_3}-\mathrm{id}_B)(b_1+\psi_{c_1}(b_2)) + (\phi_{c_3}\psi_{c_1c_2})(b_3),c_3\big)\]
on the other. Note that the terms involving $b_3$ are always equal. By comparing the terms involving $b_1$ and $b_2$, respectively, we see that $(B\times C,+,\circ)$ is $\lambda$-homomorphic if and only if both
\[ (\phi_{c_3}-\mathrm{id}_B)(\phi_{c_2}^{-1}-\mathrm{id}_B) =0\quad\textnormal{and}\quad  (\phi_{c_3}-\mathrm{id}_B)(\mathrm{id}_B-\psi_{c_1}) = 0\]
always hold, and the claim (c) now follows.
\end{proof}

\begin{theorem}\label{thm:second method}
Let $B=(B,+)$ be an abelian group and $C=(C,\cdot)$ a group. Given any homomorphisms $\phi,\gamma,\psi : c\in C\longmapsto \phi_c,\gamma_c,\psi_c\in\Aut(B)$, define
\[ (b_1,c_1) + (b_2,c_2) = \big(b_1 + \phi_{c_1}(b_2),c_1c_2\big)\]
and 
\[ \big((\phi_{c_1}\gamma_{c_1})(b_1),c_1\big) \circ \big((\phi_{c_2}\gamma_{c_2})(b_2),c_2\big) = \big( (\phi_{c_1c_2}\gamma_{c_1c_2})(\psi_{c_2}^{-1}(b_1)+b_2),c_1c_2\big)\]
for all $b_1,b_2\in B$ and $c_1,c_2\in C$. 

\smallskip

Then $(B\times C,+,\circ)$ is a skew brace if and only if the relations
\begin{equation}\label{eq:second method} \phi_{c}\gamma_{c'} = \gamma_{c'}\phi_{c}\quad\textnormal{and}\quad
\phi_{c'}(\gamma_{cc'}\psi_{cc'}^{-1} - \gamma_{c'}\psi_{c'}^{-1}) =\\\gamma_c\psi_c^{-1}-\mathrm{id}_B\end{equation}
hold for all $c,c'\in C$. In this case, we have:
\begin{enumerate}[$(a)$]
\item Assuming that $\mathrm{Im}(\gamma)$ is abelian, we have that $(B\times C,+,\circ)$ is two-sided if and only if
\[ 
\gamma_c\psi_{c'} =\psi_{c'}\gamma_c,\quad
\phi_{c}(\phi_{c'}\psi_{c'}^{-1})=(\phi_{c'}\psi_{c'}^{-1})\phi_{c},\quad \textnormal{and}
\]
\[(\phi_c\gamma_c-\mathrm{id}_B)(\gamma_{c'}-\mathrm{id}_B)=0 \]
hold for all $c,c'\in C$.
\item Assuming that $\mathrm{Im}(\gamma)$ is abelian, we have that $(B\times C,+,\circ)$ is a bi-skew brace if and only if
\[ 
\gamma_c\psi_{c'}=\psi_{c'}\gamma_c\quad\textnormal{and}\quad
(\phi_c\gamma_c-\mathrm{id}_B)(\phi_{c'}\gamma_{c'}\psi_{c'}^{-1}-\mathrm{id}_B)=0
\]
hold for all $c,c'\in C$.
\item $(B\times C,+,\circ)$ is $\lambda$-homomorphic if and only if
\[ (\phi_c\gamma_c\psi_c^{-1}-\mathrm{id}_B)(\phi_{c'}\gamma_{c'}\psi_{c'}^{-1}-\mathrm{id}_B)=0,\quad\textnormal{and}\] \[(\phi_c\gamma_c\psi_c^{-1}-\mathrm{id}_B)(\gamma_{c'}-\mathrm{id}_B)=0\]
 hold for all $c,c'\in C$. 
\end{enumerate}
Moreover, the following hold:
\begin{enumerate}[$(a)$]\setcounter{enumi}{+3}
    \item For any $b\in B$ and $c\in C$, we have $( (\phi_c\gamma_c)(b),c)\in \Ker(\lambda)$ if and only if
    \[ b \in \bigcap_{y\in C} \Ker(\phi_y\gamma_y\psi_y^{-1}-\mathrm{id}_B)\quad\textnormal{and}\quad c\in \Ker(\gamma).\]
    \item For any subgroup $I$ of $B$, the subset $I\times \{1\}$ is an ideal of $(B\times C,+,\circ)$ if and only if $I$ is invariant under $\phi, \gamma,\psi$, namely if and only if
\begin{equation}\label{eq:ideal conditions} \phi_c(I)\subseteq I,\quad \gamma_c(I)\subseteq I,\quad \psi_c(I)\subseteq I,\end{equation}
for all $c\in C$.
\end{enumerate}
\end{theorem}
\begin{proof} Clearly, $(B\times C,+)$ is a group, while $(B\times C,\circ)$ is a group because $\circ$ is the operation on $B\times C$ induced via transport by the bijection
\[ (b,c)\in B\times C\longmapsto \big((\gamma_c^{-1}\phi_c^{-1})(b),c^{-1}\big)\in (B\rtimes_\psi C)^{\mbox{\tiny op}},\]
where the superscript ${}^{\mbox{\tiny op}}$ denotes the opposite group. In what follows, define
\[ a_1 = ((\phi_{c_1}\gamma_{c_1})(b_1), c_1),\quad 
a_2=((\phi_{c_2}\gamma_{c_2})(b_2), c_2),\quad
a_3=((\phi_{c_3}\gamma_{c_3})(b_3), c_3)\]
with $b_1,b_2,b_3\in B$ and $c_1,c_2,c_3\in C$. 

\smallskip

For the skew left distributivity, observe that
\begin{align*}
&\hspace{5mm} a_1\circ (a_2+a_3)\\
&= a_1\circ \big( (\phi_{c_2}\gamma_{c_2})(b_2) + (\phi_{c_2c_3}\gamma_{c_3})(b_3),c_2c_3\big)\\
& = \big((\phi_{c_1c_2c_3}\gamma_{c_1c_2c_3})(\psi_{c_2c_3}^{-1}(b_1) + (\gamma_{c_2c_3}^{-1}\phi_{c_3}^{-1}\gamma_{c_2})(b_2) + \gamma_{c_3^{-1}c_2c_3}^{-1}(b_3)),c_1c_2c_3\big)
\end{align*}
on the one hand, and
\begin{align*}
&\hspace{5mm}
(a_1\circ a_2) - a_1 + (a_1\circ a_3)\\
&= \big((\phi_{c_1c_2}\gamma_{c_1c_2})(\psi_{c_2}^{-1}(b_1)+b_2),c_1c_2\big) + (\gamma_{c_1}(-b_1),c_1^{-1}) + (a_1\circ a_3)\\
&= \big( (\phi_{c_1c_2}\gamma_{c_1c_2})(\psi_{c_2}^{-1}(b_1)+b_2)
- (\phi_{c_1c_2}\gamma_{c_1})(b_1),c_1c_2c_1^{-1}\big) + (a_1\circ a_3)\\
&=  \big( (\phi_{c_1c_2}\gamma_{c_1c_2})(\psi_{c_2}^{-1}(b_1)+b_2)
- (\phi_{c_1c_2}\gamma_{c_1})(b_1)   \\
&\hspace{5cm}+(\phi_{c_1c_2c_3}\gamma_{c_1c_3})(\psi_{c_3}^{-1}(b_1)+b_3),c_1c_2c_3\big)
\end{align*}
on the other. Note that the terms involving $b_3$ are always equal. By comparing the terms with $b_2$ and $b_1$, respectively, we see that $(B\times C,+,\circ)$ is a skew brace if and only if the relations
\[ \phi_{c_3}\gamma_{c_1}\phi_{c_3}^{-1} = \gamma_{c_1}\quad\textnormal{and}\quad \phi_{c_3}\gamma_{c_1c_2c_3}\psi_{c_2c_3}^{-1} = \gamma_{c_1c_2}\psi_{c_2}^{-1} -\gamma_{c_1} + \phi_{c_3}\gamma_{c_1c_3}\psi_{c_3}^{-1}\]
always hold. Using the former, we can simplify the latter to
\[ \phi_{c_3} (\gamma_{c_2c_3}\psi_{c_2c_3}^{-1}-\gamma_{c_3}\psi_{c_3}^{-1}) = \gamma_{c_2}\psi_{c_2}^{-1}-\mathrm{id}_B ,\]
and so indeed \eqref{eq:second method} is  both sufficient and necessary for $(B\times C,+,\circ)$ to be a skew brace.

\smallskip

Now, suppose that \eqref{eq:second method} holds. In (a) and (b), we shall assume further that $\mathrm{Im}(\gamma)$ is abelian.

\smallskip

\noindent\underline{Proof of (a)}

\smallskip

\noindent For the skew right distributivity, we have (using \eqref{eq:second method})
\begin{align*}
&\hspace{5mm}(a_1+a_2)\circ a_3\\
& = \big( (\phi_{c_1}\gamma_{c_1})(b_1) +(\phi_{c_1c_2}\gamma_{c_2})(b_2),c_1c_2\big)\circ a_3\\
& = \big( (\phi_{c_1c_2c_3}\gamma_{c_1c_2c_3})(\psi_{c_3}^{-1}( (\phi_{c_2}^{-1}\gamma_{c_2}^{-1})(b_1) + (\gamma_{c_1c_2}^{-1}\gamma_{c_2})(b_2)) + b_3),c_1c_2c_3\big)
\end{align*}
on the one hand, and
\begin{align*}
&\hspace{5mm} (a_1\circ a_3)-a_3+(a_2\circ a_3)\\
& =  \big( (\phi_{c_1c_3}\gamma_{c_1c_3})(\psi_{c_3}^{-1}(b_1) + b_3),c_1c_3\big) + (\gamma_{c_3}(-b_3),c_3^{-1})+(a_2\circ a_3)\\
&= \big( (\phi_{c_1c_3}\gamma_{c_1c_3})(\psi_{c_3}^{-1}(b_1)+b_3) -  (\phi_{c_1c_3}\gamma_{c_3})(b_3),c_1\big)+(a_2\circ a_3)\\
& = \big( (\phi_{c_1c_3}\gamma_{c_1c_3})(\psi_{c_3}^{-1}(b_1)+b_3) -  (\phi_{c_1c_3}\gamma_{c_3})(b_3)  \\
&\hspace{5cm}+(\phi_{c_1c_2c_3}\gamma_{c_2c_3})( \psi_{c_3}^{-1}(b_2)+b_3),c_1c_2c_3\big)
\end{align*}
on the other. Since $\mathrm{Im}(\gamma)$ is abelian here, by comparing the terms with $b_2$, we see that it is necessary that
\[\gamma_{c_1} \psi_{c_3}^{-1} \gamma_{c_1}^{-1} =  \psi_{c_3}^{-1}.\]
Without loss of generality, we may assume that this equality always holds. Keeping \eqref{eq:second method} and the fact that $\mathrm{Im}(\gamma)$ is abelian in mind, by further comparing the terms involving $b_1$ and $b_3$, respectively, we see that $(B\times C,+,\circ)$ is two-sided if and only if the relations
\begin{align*}
 \phi_{c_2}(\phi_{c_3}\psi_{c_3}^{-1})\phi_{c_2}^{-1} &= \phi_{c_3}\psi_{c_3}^{-1},\quad \textnormal{and}\\
\phi_{c_2}\gamma_{c_1c_2} & = (\gamma_{c_1}-\mathrm{id}_B) + \phi_{c_2}\gamma_{c_2}
\end{align*}
also always hold. It follows that the claim (a) is true.

\smallskip

\noindent \underline{Proof of (b)}

\smallskip

\noindent For the skew left distributivity with $+$ and $\circ$ reversed, we have
\begin{align*}
&\hspace{5mm}a_1 + (a_2\circ a_3)\\
& = a_1 + \big((\phi_{c_2c_3}\gamma_{c_2c_3})(\psi_{c_3}^{-1}(b_2)+b_3),c_2c_3\big)\\
&= \big((\phi_{c_1}\gamma_{c_1})(b_1) + (\phi_{c_1c_2c_3}\gamma_{c_2c_3})(\psi_{c_3}^{-1}(b_2)+b_3),c_1c_2c_3\big)
\end{align*}
on the one hand, and (using \eqref{eq:second method})
\begin{align*}
&\hspace{5mm}(a_1+a_2)\circ a_1^{-1}\circ (a_1+a_3)\\
& = \big((\phi_{c_1}\gamma_{c_1})(b_1) + (\phi_{c_1c_2}\gamma_{c_2})(b_2),c_1c_2\big)\circ \big((\phi_{c_1}^{-1}\gamma_{c_1}^{-1})(\psi_{c_1}(-b_1)),c_1^{-1}\big) \circ (a_1+a_3)\\
& = \big((\phi_{c_1c_2c_1^{-1}}\gamma_{c_1c_2c_1^{-1}})(\psi_{c_1}( (\phi_{c_2}^{-1}\gamma_{c_2}^{-1})(b_1) + (\gamma_{c_1c_2}^{-1}\gamma_{c_2})(b_2)-b_1), c_1c_2c_1^{-1}\big) \circ (a_1+a_3)\\
&= \big( (\phi_{c_1c_2c_3}\gamma_{c_1c_2c_3})(\psi_{c_3}^{-1}( (\phi_{c_2}^{-1}\gamma_{c_2}^{-1}-\mathrm{id}_B)(b_1) + (\gamma_{c_1c_2}^{-1}\gamma_{c_2})(b_2)) \\
&\hspace{5cm}+ (\phi_{c_3}^{-1}\gamma_{c_3}^{-1})(b_1) + (\gamma_{c_1c_3}^{-1}\gamma_{c_3})(b_3)),c_1c_2c_3\big) 
\end{align*}
on the other. Since $\mathrm{Im}(\gamma)$ is abelian here, the terms involving $b_3$ are always equal. Keeping \eqref{eq:second method} in mind, by comparing the terms involving $b_1$ and $b_2$, respectively, we see that $(B\times C,+,\circ)$ is a bi-skew brace if and only if
\begin{align*}
\mathrm{id}_B & = \phi_{c_2c_3}\gamma_{c_2c_3}\big(\psi_{c_3}^{-1}(\phi_{c_2}^{-1}\gamma_{c_2}^{-1}-\mathrm{id}_B) + \phi_{c_3}^{-1}\gamma_{c_3}^{-1}\big),\quad\textnormal{and}\\
\psi_{c_3}^{-1} & = \gamma_{c_1}\psi_{c_3}^{-1}\gamma_{c_1}^{-1}
\end{align*}
always hold. Using \eqref{eq:second method}, the former relation can be rearranged to
\[ \phi_{c_3}^{-1}\gamma_{c_3}^{-1}(\phi_{c_2}^{-1}\gamma_{c_2}^{-1}-\mathrm{id}_B) =\psi_{c_3}^{-1}(\phi_{c_2}^{-1}\gamma_{c_2}^{-1}-\mathrm{id}_B),\]
and we see that the claim (b) holds.

\smallskip

\noindent\underline{Proof of (c)}

\smallskip

\noindent For any $b,x\in B$ and $c,y\in C$, we have (using \eqref{eq:second method})
\begin{align}\label{eq:second lambda}
&\lambda_{((\phi_c\gamma_c)(b),c)}((\phi_y\gamma_y)(x),y) \\\notag
&\hspace{2cm} = (\gamma_c(-b),c^{-1})+ \big( (\phi_{cy}\gamma_{cy})(\psi_y^{-1}(b) + x),cy\big)\\\notag
&\hspace{2cm}= \big( \gamma_c(\phi_y\gamma_y\psi_y^{-1}-\mathrm{id}_B)(b) + (\phi_y\gamma_{cy})(x),y\big).
\end{align}
Using this and again \eqref{eq:second method}, we get
\begin{align*}
\lambda_{a_1+a_2}(a_3) 
&= \big(\gamma_{c_1c_2}(\phi_{c_3}\gamma_{c_3}\psi_{c_3}^{-1}-\mathrm{id}_B)( (\phi_{c_2}^{-1}\gamma_{c_2}^{-1})(b_1)+(\gamma_{c_1c_2}^{-1}\gamma_{c_2})(b_2))\\
&\hspace{6.5cm}  + (\phi_{c_3}\gamma_{c_1c_2c_3})(b_3),c_3\big)
\end{align*}
on the one hand, and
\begin{align*}
\lambda_{a_1\circ a_2}(a_3) & = \big(\gamma_{c_1c_2}(\phi_{c_3}\gamma_{c_3}\psi_{c_3}^{-1}-\mathrm{id}_B)(\psi_{c_2}^{-1}(b_1)+b_2)\\
&\hspace{4cm} + (\phi_{c_3}\gamma_{c_1c_2c_3})(b_3),c_3\big)
\end{align*}
on the other. The terms involving $b_3$ are always equal. Thus, by comparing the terms involving $b_1$ and $b_2$, respectively, we see that $(B\times C,+,\circ)$ is $\lambda$-homomorphic if and only if both
\begin{align*}(\phi_{c_3}\gamma_{c_3}\psi_{c_3}^{-1} -\mathrm{id}_B)(\phi_{c_2}^{-1}\gamma_{c_2}^{-1}\psi_{c_2}-\mathrm{id}_B)&=0,\quad\textnormal{and}\\
(\phi_{c_3}\gamma_{c_3}\psi_{c_3}^{-1} -\mathrm{id}_B)(\gamma_{c_2c_1^{-1}c_2^{-1}}-\mathrm{id}_B) & =0
\end{align*}
always hold, and the claim (c) now follows.

\smallskip

\noindent\underline{Proof of (d)}

\smallskip

\noindent It follows from \eqref{eq:second lambda} that $((\phi_c\gamma_c)(b),c)\in \Ker(\lambda)$ if and only if
\[ \gamma_c(\phi_y\gamma_y\psi_y^{-1}-\mathrm{id}_B)(b) + \phi_y(\gamma_{c} - \mathrm{id}_B)\gamma_y(x) =0\]
for all $x\in B$ and $y\in C$. Taking $x=0$ and $y=1$, respectively, we see that 
\[ (\phi_y\gamma_y\psi_y^{-1} - \mathrm{id}_B)(b) =0\quad\textnormal{and}\quad  (\gamma_c-\mathrm{id}_B)(x) =0\]
for all $y\in C$ and $x\in B$, namely
\[ ( b,c)\in \left(\bigcap_{y\in C}\Ker(\phi_y\gamma_y\psi_y^{-1}-\mathrm{id}_B)\right)\times \Ker(\gamma),\]
whenever $((\phi_c\gamma_c)(b),c)\in\Ker(\lambda)$. The converse is clear, and this proves the claim (d).

\smallskip

\noindent\underline{Proof of (e)}

\smallskip

\noindent Clearly, $I\times \{1\}$ is a subgroup of $(B\times C,+)$, and since $B$ is an abelian group, the first inclusion in \eqref{eq:ideal conditions} is equivalent to
\begin{enumerate}[(1)]
    \item $I\times \{1\}$ being a normal subgroup of $(B\times C,+)$.
\end{enumerate}
For any $b,x\in B$ and $c\in C$, note that
\begin{align*}
\lambda_{((\phi_c\gamma_c)(b),c)}(x,1) 
&=(\gamma_c(x),1)
\end{align*}
by \eqref{eq:second lambda}, and a straightforward calculation shows that
\begin{align*}
 \big((\phi_c\gamma_c)(b),c\big)\circ (x,1)\circ \big((\phi_c\gamma_c)(b),c\big)^{-1} 
 & = (\psi_c(x),1).
\end{align*}
Hence, the second and third inclusions in \eqref{eq:ideal conditions} are equivalent to
\begin{enumerate}[(1)]\setcounter{enumi}{+1}
    \item $I\times \{1\}$ being a left-ideal of $(B\times C,+,\circ)$, and
    \item $I\times \{1\}$ being normal in $(B\times C,\circ)$,
\end{enumerate}
respectively. Then, (1), (2), and (3) above are exactly the conditions that we need for $I\times \{1\}$ to be an ideal of $(B\times C,+,\circ)$.
\end{proof}

We shall now apply Theorems \ref{thm:first method} and \ref{thm:second method} to construct examples of skew braces of order $pq$ and $p^2q$ that do not satisfy certain algebraic properties, where $p,q$ are distinct primes satisfying a certain divisibility condition.

\begin{example}\label{ex:pq1}Let $p,q$ be any primes with $q\mid p^2-1$.
\begin{enumerate}[(a)]
\item If $q\mid p-1$, take $B = (\mathbb{F}_p,+)$ and $C = (\mathbb{F}_q,+)$.
\item If $q\nmid p-1$, take $B = (\mathbb{F}^2_p,+)$ and $C = (\mathbb{F}_q,+)$.
\end{enumerate}
In both cases, we can find $\beta\in \Aut(B)$ of order $q$, and note that $\beta - \mathrm{id}_B$ is invertible. Indeed, in case (a) this is clear, and in case (b) this is because $1$ cannot be an eigenvalue of $\beta$ by the condition $q\nmid p-1$.
\begin{enumerate}[(i)]
\item Consider the homomorphisms
\begin{align*}
    \phi : c \in C\longmapsto \phi_c=\mathrm{id}_B\in\mathrm{Aut}(B)
\end{align*}
and
\begin{align*}
    \psi : c \in C\longmapsto \psi_c=\beta^c\in\mathrm{Aut}(B),
\end{align*}
which clearly satisfy \eqref{eq:first method}. We then obtain a skew brace $(B\times C,+,\circ)$ from Theorem \ref{thm:first method}. Since
\[ (\phi_1\psi_1-\mathrm{id}_B)(\psi_1-\mathrm{id}_B) = (\beta-\mathrm{id}_B)^2\]
is not zero, this skew brace is not two-sided.
\item Consider the homomorphisms
\begin{align*}
    \phi : c \in C&\longmapsto \phi_c=\beta^c\in\mathrm{Aut}(B)
\end{align*}
and
\begin{align*}
    \psi : c \in C&\longmapsto \psi_c=\mathrm{id}_B\in\mathrm{Aut}(B),
\end{align*}
which clearly satisfy \eqref{eq:first method}. We then obtain a skew brace $(B\times C,+,\circ)$ from Theorem \ref{thm:first method}. Since
\begin{align*}
 (\phi_1\psi_1-\mathrm{id}_B)(\phi_1-\mathrm{id}_B) &= (\beta-\mathrm{id}_B)^2\\
(\phi_1-\mathrm{id}_B)(\phi_1-\mathrm{id}_B) & = (\beta-\mathrm{id}_B)^2
 \end{align*}
are not zero, respectively, this skew brace is not a bi-skew brace and is not $\lambda$-homomorphic.
\end{enumerate}
\end{example}

\begin{example}\label{ex:pq2}Let $p,q$ be any primes with $q\mid p-1$, and let $g$ be an integer of multiplicative order $q$ modulo $p$. Take $B = (\mathbb{F}_p^2,+)$ and $C = (\mathbb{F}_q,+)$. Here $\left[\begin{smallmatrix}
        1 & 1  \\  1& -1 \end{smallmatrix}\right]$
is invertible because~$p$ is odd. Let us consider the homomorphisms
\begin{align*}
    \phi : c \in C&\longmapsto \Big[\begin{smallmatrix}
        g^{-1} &  \\  & g 
    \end{smallmatrix} \Big]^c\in\mathrm{Aut}(B),\\
    \gamma : c \in C&\longmapsto \Big[\begin{smallmatrix}
        g &  \\  & 1 
    \end{smallmatrix} \Big]^c\in\mathrm{Aut}(B),\quad\textnormal{and}\\
    \psi : c\in C & \longmapsto \Big[\begin{smallmatrix}
        1 & 1  \\  1& -1 
    \end{smallmatrix}\Big]\Big[\begin{smallmatrix}
        g &  \\  & 1 
    \end{smallmatrix} \Big]^c\Big[\begin{smallmatrix}
        1 & 1  \\  1& -1 \end{smallmatrix}\Big]^{-1}\in\mathrm{Aut}(B),
\end{align*}
which satisfy the first relation in \eqref{eq:second method}. For any $c\in \mathbb{F}_q$, observe that
\[ \gamma_c\psi_{c}^{-1}= \frac{1}{2}\begin{bmatrix}
        g^c+1 & -g^c+1\\ g^{-c}-1 & g^{-c}+1
    \end{bmatrix}.\]
Using this, it is not hard to check that the second relation in \eqref{eq:second method} also holds. We then obtain a skew brace $(B\times C,+,\circ)$ from Theorem \ref{thm:second method}. This skew brace clearly has no ideal of order $q$, and by \eqref{eq:ideal conditions}, it has no ideal of order $p$ either because $\gamma_1,\, \psi_1$ have no eigenvectors in common. In particular, this skew brace is not supersoluble.
\end{example}

Finally, we shall apply Theorem \ref{thm:second method} to construct examples of skew braces of order $p^3$ that do not possess certain algebraic properties, where $p$ is any odd prime.

\begin{example}\label{ex:p^3}Let $p$ be any odd prime. Take $B = (\mathbb{F}_p^2,+)$ and $C = (\mathbb{F}_p,+)$. Here $\left[\begin{smallmatrix}
        1 & 1  \\  1& -1 \end{smallmatrix}\right]$
is invertible since $p$ is odd. Consider the homomorphisms
\begin{align*}
    \phi : c \in C&\longmapsto \phi_c=\Big[\begin{smallmatrix}
        1 & -2c  \\  & 1 
    \end{smallmatrix} \Big]\in\mathrm{Aut}(B),\\
    \gamma : c \in C&\longmapsto \gamma_c=\Big[\begin{smallmatrix}
        1 & c \\  & 1 
    \end{smallmatrix} \Big]\in\mathrm{Aut}(B),\quad\textnormal{and}\\
    \psi : c\in C & \longmapsto \psi_c=\Big[\begin{smallmatrix}
        1 & 1  \\  1& -1 \end{smallmatrix}\Big]\Big[\begin{smallmatrix}
        1 & c \\  & 1 
    \end{smallmatrix} \Big]\Big[\begin{smallmatrix}
        1 & 1  \\  1& -1 \end{smallmatrix}\Big]^{-1}\in\mathrm{Aut}(B),
\end{align*}
which clearly satisfy the first relation in (\ref{eq:second method}). For any $c\in \mathbb{F}_p$, we have
\[ \gamma_c\psi_c^{-1}
=\frac{1}{2}\begin{bmatrix}
   -c^2-c+2 & c^2+3c\\ -c & c+2 
\end{bmatrix}.\]
Using this, one can check that the second relation in \eqref{eq:second method} also holds. We then obtain a skew brace $(B\times C,+,\circ)$ from Theorem \ref{thm:second method}. It is easy to check that $\gamma_1$ and $\psi_1$ do not commute, so this skew brace is not two-sided and not a bi-skew brace. Note also that
\[ (\phi_1\gamma_1\psi_1^{-1}-\mathrm{id}_B)(\gamma_1-\mathrm{id}_B) =\left[\begin{smallmatrix} 0 & 0 \\ 0 & -\frac{1}{2}\end{smallmatrix}\right]\]
is not zero, so this skew brace is not $\lambda$-homomorphic either. Moreover, note that $\Ker(\gamma)=\{0\}$, and that
\[ \Ker(\phi_1\gamma_1\psi_1^{-1}-\mathrm{id}_B) =\Ker\left[\begin{smallmatrix} 0 & -1 \\ -\frac{1}{2} & \frac{1}{2}\end{smallmatrix}\right]=\{0\}.\]
It follows that $\Ker(\lambda)=\{0\}$ for this skew brace, so in particular
\[ \mathrm{Soc}(B\times C,+,\circ) = \{0\}\]
and hence the skew brace is not of finite multipermutation level.
\end{example}

\section{Proof of Theorem A}

\begin{theorem}\label{thm:cyclic}
Let $n$ be a natural number, and let $p_1^{\alpha_1}\ldots p_t^{\alpha_t}$ be its prime factorisation. Then the following are equivalent:
\begin{enumerate}[$(1)$]
\item Every sb-solution of cardinality $n$ is a flip solution.
\item Every skew brace of order $n$ is a trivial brace.
\item Every skew brace of order $n$ is trivial.
\item Every skew brace of order $n$ is almost trivial.
\item Every skew brace of order $n$ is weakly trivial.
\item Every skew brace of order $n$ is one-generator.
\item $\alpha_i=1$ for every $i$, and $p_i$ does not divide $p_j-1$ for $i\neq j$.
\end{enumerate}
\end{theorem}

Note that condition (7) is equivalent to requiring that all groups of order~$n$ are cyclic (this is a well-known fact; see \cite[Theorem 3.0.4]{thesissuper}).

\begin{proof} Since the flip solutions correspond to trivial braces under \eqref{rA}, the equivalence of (1) and (2) is clear. Note also that in case of a brace, the notions of ``trivial", ``almost trivial", and ``weakly trivial" coincide.

\smallskip

\noindent\underline{Proof of (2) $\sim$ (5) $\Rightarrow$ (7)}

\smallskip

For any prime $p$, braces of order $p^2$ were classified in \cite{p3}, and there is a non-trivial brace of order $p^2$.

\smallskip

For any primes $p,q$ with $q\mid p-1$, skew braces of order $pq$ were classified in \cite{pq}, and there is a non-trivial brace of order $pq$.

\smallskip

Therefore, if $n$ satisfies any one of (2) $\sim$ (5), then $n$ must also satisfy the conditions given in (7). 

\smallskip

\noindent\underline{Proof of (6) $\Rightarrow$ (7)}

\smallskip

This holds because a trivial skew brace is one-generator if and only if the underlying group is cyclic.

\smallskip

\noindent\underline{Proof of (7) $\Rightarrow$ (2) $\sim$ (6)}

\smallskip

Now, suppose that $n$ satisfies (7), and let $(A,+,\circ)$ be any skew brace of order $n$. Then $(A,+)$ and $(A,\circ)$ are both cyclic, so by Proposition \ref{prop:nilpotent}, we may assume that $n = p$ is a prime. 
But then obviously $(A,+,\circ)$ is a trivial brace. Since $(A,+)$ is cyclic, clearly~$(A,+,\circ)$ is also one-generator.
\end{proof}

\section{Proof of Theorem B}

\begin{theorem}\label{thm:abelian} 
Let $n$ be a natural number, and let $p_1^{\alpha_1}\ldots p_t^{\alpha_t}$ be its prime factorisation. Then the following are equivalent:
\begin{enumerate}[$(1)$]
\item Every sb-solution of cardinality $n$ is multipermutation.
\item Every sb-solution of cardinality $n$ is involutive.
\item Every skew brace of order $n$ has finite multipermutation level.
\item Every skew brace of order $n$ is right-nilpotent.
\item Every skew brace of order $n$ is annihilator nilpotent.
\item Every skew brace of order $n$ is a brace.
\item Every skew brace of order $n$ is two-sided.
\item Every skew brace of order $n$ is a bi-skew brace.
\item Every skew brace of order $n$ is $\lambda$-homomorphic.
\item $\alpha_i\leq2$ for every $i$, and $p_i$ does not divide $p_j^{\alpha_j}-1$ for $i\neq j$.
\end{enumerate}
\end{theorem}

 Note that condition (10) is equivalent to requiring that all groups of order~$n$ are abelian (this is due to Dickson \cite{Dickson}, or see \cite[Theorem 4.3.1]{thesissuper}).
 
\begin{proof}

The equivalence of (1) and (3) follows from Theorem \ref{thm:fml}. Since the involutive solutions correspond to braces under \eqref{rA}, as is known by \cite{Rump0}, the equivalence of (2) and (6) is also clear.

\smallskip

\noindent\underline{Proof of (3) $\Leftrightarrow$ (4) $\Leftrightarrow$ (5)}

\smallskip

Note that for almost trivial skew braces, the notions of ``finite multipermutation level", ``right-nilpotent", and ``annihilator nilpotent" coincide with the underlying group being nilpotent. Thus, it suffices to consider the natural numbers $n$ for which every group of order $n$ is nilpotent. But then for any skew brace $(A,+,\circ)$ of order $n$, since $(A,+)$ and $(A,\circ)$ are both nilpotent, the properties of being of finite multipermutation level, right-nilpotent, and annihilator nilpotent are equivalent by~\cite[The\-o\-rem~2.20]{nilpotent} and~\cite[Co\-rol\-la\-ry~2.11]{bonatto}. This shows that (3), (4), and (5) are equivalent.

\smallskip

\noindent\underline{Proof of (6) $\Leftrightarrow$ (10)}
\smallskip

The forward implication holds by considering trivial skew braces, and the backward implication is obvious.

\smallskip

\noindent\underline{Proof of (3),\, (7) $\sim$ (9) $\Rightarrow$ (10)}

\smallskip

For any odd prime $p$, the skew brace of order $p^3$ in Example \ref{ex:p^3} is not two-sided, not a bi-skew brace, not $\lambda$-homomorphic, and not of finite multipermutation level. For $p=2$, one can check using the \texttt{YangBaxter} package in \texttt{GAP} \cite{GAP} that \texttt{SmallSkewbrace(8,47)} is not two-sided, not a bi-skew brace, and not of finite multipermutation level. It is also not $\lambda$-homomorphic because the kernel of its $\lambda$-map does not contain the derived ideal. 

\smallskip

For any primes $p,q$ with $q\mid p^2-1$, the skew braces of order $pq$ (in case $q\mid p-1$) and $p^2q$ (in case $q\nmid p-1$) in Example \ref{ex:pq1}(i) are not two-sided, while those in Example \ref{ex:pq1}(ii) are not bi-skew braces and not $\lambda$-homomorphic. We also have a trivial skew brace of order $pq$ (in case $q\mid p-1$) and $p^2q$ (in case $q\nmid p-1$) that is not of finite multipermutation level because there is a non-nilpotent group of the corresponding order.

\smallskip

Therefore, if $n$ satisfies any one of (3),\, (7) $\sim$ (9), then $n$ must also satisfy the conditions given in (10).

\smallskip 

\noindent\underline{Proof of (10) $\Rightarrow$ (3),\, (7) $\sim$ (9)}

\smallskip

Now, suppose that $n$ satisfies (10), and let $(A,+,\circ)$ be any skew brace of order $n$. Then $(A,+)$ and $(A,\circ)$ are abelian, so $(A,+,\circ)$ is clearly a brace and is two-sided. Moreover, by Proposition \ref{prop:nilpotent}, we may assume that $n=p$ is a prime or $n=p^2$ is the square of a prime. In case $(A,+,\circ)$ is trivial, it clearly satisfies all of the other properties. In case $(A,+,\circ)$ is non-trivial, up to isomorphism, there are only two possibilities 
for $A$ by \cite{p3}:
\begin{enumerate}[(i)]
\item The brace $(\mathbb{Z}/p^2\mathbb{Z},+,\circ)$, where $+$ is the usual addition and
\[ a \circ b =  a + b + pab\]
for all $a,b\in \mathbb{Z}/p^2\mathbb{Z}$. Its $\lambda$-map is given by
\[ \lambda : a\in \mathbb{Z}/p^2\mathbb{Z}\longmapsto 1+pa\in (\mathbb{Z}/p^2\mathbb{Z})^\times \simeq \Aut(\mathbb{Z}/p^2\mathbb{Z}).\]
\item The brace $(\mathbb{F}_p^2,+,\circ)$, where $+$ is the usual addition and
\[  \begin{bmatrix}a_1\\a_2\end{bmatrix}
\circ \begin{bmatrix} b_1\\b_2\end{bmatrix} = \begin{bmatrix} a_1 + b_1 + a_2b_2 \\ a_2 +b_2\end{bmatrix}\]
for all $a_1,a_2,b_1,b_2\in \mathbb{F}_p$. Its $\lambda$-map is given by
\[ \lambda : \begin{bmatrix}a_1\\a_2\end{bmatrix}\in \mathbb{F}_p^2\longmapsto \begin{bmatrix}1 & a_2\\ 0 & 1 \end{bmatrix}\in \mathrm{GL}_2(\mathbb{F}_p) \simeq \Aut(\mathbb{F}_p^2).\]
\end{enumerate}
For both of the braces $A$, observe that
\[ \lambda_{a+b} = \lambda_a\lambda_b\quad\textnormal{and}\quad \lambda_{\lambda_a(b)} = \lambda_b = \lambda_a\lambda_b\lambda_a^{-1}\]
for all $a,b\in A$. It follows that $A$ is $\lambda$-homomorphic, and $A$ is a bi-skew brace by \eqref{biskewbrace}. Moreover, note  that $\mathrm{Soc}(A) = \Ker(\lambda)$ and the quotient~$A/\mathrm{Soc}(A)$ has order $p$, so  $\mathrm{Soc}_2(A)=A$ and $A$ has finite multipermutation level.
\end{proof}

\begin{theorem}\label{thm:nilpotent}
Let $n$ be a natural number, and let $p_1^{\alpha_1}\ldots p_t^{\alpha_t}$ be its prime factorisation. Then the following are equivalent:
\begin{enumerate}[$(1)$]
\item Every skew brace of order $n$ is left-nilpotent.
\item $p_i$ does not divide $p_j^k-1$ for $i\neq j$ and $1\leq k\leq\alpha_j$.
\end{enumerate}
\end{theorem}

Note that condition (2) in Theorem \ref{thm:nilpotent} is equivalent to requiring that all groups of order $n$ are nilpotent (see \cite[Theorem 5.2.3]{thesissuper} for example, or \cite{redei}).

\begin{proof} We have (1) implies (2) because an almost trivial skew brace is left-nilpotent if and only if its underlying group is nilpotent. Conversely, suppose that $n$ satisfies (2), and let $(A,+,\circ)$ be any skew brace of order $n$. Then~$(A,+)$ and $(A,\circ)$ are nilpotent, so by Proposition \ref{prop:nilpotent}, we may assume that~$n$ is a prime power. But skew braces of prime power order are left-nilpotent by \cite[Proposition 4.4]{nilpotent}, and this proves (1).
\end{proof}

\section{Proof of Theorem C}

\begin{theorem}\label{super}
Let $n$ be a natural number, and let $n=p_1^{\alpha_1}\ldots p_t^{\alpha_t}$ be its prime factorisation. Then the following are equivalent:
\begin{enumerate}[$(1)$]
    \item Every skew brace of order $n$ is supersoluble.
    \item $
    \alpha_i\leq2$ for every $i$, and the following two conditions hold:
    \begin{enumerate}[$\bullet$]
    \item If $\alpha_j=2$, then $p_i$ does not divide $p_j^2-1$ for $i\neq j$; 
    \item If $4$ divides $n$, then $p_i\equiv1\pmod{4}$ for every $i$ with $\alpha_i=2$. 
    \end{enumerate} 
\end{enumerate}
In particular, if $n$ satisfies condition (2), then every sb-solution of cardinality~$n$ is supersoluble.
\end{theorem}

 Note that condition (2) is sufficient to guarantee that all groups of order~$n$ are supersoluble (see \cite[Theorem 7.3.1]{thesissuper} for example).

\begin{proof}The last sentence follows from Theorem \ref{thm:supersoluble solution}.

\smallskip

\noindent\underline{Proof (1) $\Rightarrow $ (2)}

\smallskip

By \cite[Theorem 3.7]{supersoluble}, every finite supersoluble skew brace of prime power order is annihilator nilpotent, and so $n$ must be cube-free by Theorem~\ref{thm:abelian}.

\smallskip

For any primes $p,q$ with $q\mid p-1$, the skew brace of order $p^2q$ in Example~\ref{ex:pq2} is not supersoluble. For any primes $p,q$ with $q\mid p+1$, there is a trivial skew brace of order $p^2q$ that is not supersoluble because there is a Frobenius group of order $p^2q$, which is clearly not supersoluble. 

\smallskip

For any prime $p$ with $p\equiv 3\pmod{4}$, similarly there is a group and hence a trivial skew brace of order $4p^2$ that is not supersoluble. Indeed, we have the semidirect product $\mathbb{F}_p^2\rtimes \mathbb{Z}/4\mathbb{Z}$, where the generator of $\mathbb{Z}/4\mathbb{Z}$ acts on $\mathbb{F}_p^2$ via the matrix $\left[\begin{smallmatrix} 0 & -1 \\ 1 & 0 \end{smallmatrix}\right]$. Since $p\equiv 3\pmod{4}$ implies that $-1$ is not a quadratic residue mod $p$, this action is irreducible. It follows that $\mathbb{F}_p^2\rtimes\mathbb{Z}/4\mathbb{Z}$ has no normal subgroup of order $p$ and so is not supersoluble.

\smallskip

We conclude that $n$ must satisfy the conditions stated in (2).

\smallskip

\noindent\underline{Proof (2) $\Rightarrow $ (1)}

\smallskip

We use induction on the number $t$ of prime divisors of $n$. The case $t=1$ is clear. Indeed, braces of prime or prime-square order were classified in \cite{p3}, and up to isomorphism, there are only two non-trivial ones (see the proof of~Theorem \ref{thm:abelian}). For both of the possibilities, we have the series of ideals
\[ \{0\} \leq \mathrm{Soc}(A)\leq A\]
for which the consecutive factors have prime order, so they are supersoluble. Suppose then that $t\geq 2$ and let $p$ denote the largest prime divisor of $n$.

\smallskip

First, let $G$ be any group of order $n$. Since $G$ is necessarily supersoluble, it has a normal Sylow $p$-subgroup $P$ (see \cite[(5.4.8)]{Robinson}). Since $G$ is also soluble, it has a Hall $p'$-subgroup $H$ (which is unique up to conjugation) by a famous theorem of Hall (see \cite[(9.1.7)]{Robinson}). Clearly we have $G = P\rtimes H$. Moreover, note that when $|P| = p^2$, condition~(2) implies that the conjugation action of $H$ on $P$ must be trivial, meaning that $G = P\times H$ is in fact a direct product. 

\smallskip

Now, let $(A,+,\circ)$ be any skew brace of order $n$. The above implies that $(A,+)$ has a normal (and hence characteristic) Sylow $p$-subgroup $P$. This means that $P$ is a left-ideal of $A$ and in particular a subgroup of $(A,\circ)$. But then $P$ is also a Sylow $p$-subgroup of $(A,\circ)$, which is again normal by the previous paragraph. Thus, $P$ is an ideal of $A$. 
\begin{enumerate}[$\bullet$]
\item If $|P|=p$, then by induction $A/P$ is supersoluble, and we see that $A$ is also supersoluble.
\item If $|P|=p^2$, then $(A,+)$ also has a normal (and hence characteristic) Hall~\hbox{$p'$-sub}\-group $H$, and the same argument as above shows that $H$ is in fact an ideal of $A$. Therefore, we have a direct product decomposition
\[ (A,+,\circ) = (P,+,\circ) \times (H,+,\circ)\]
of the skew brace $A$ via ideals $P$ and $H$. By induction, we know that $P$ and $H$ are both supersoluble, whence $A$ is also supersoluble.
\end{enumerate}
This completes the proof.
\end{proof}

\medskip

We leave the following as an open problem. The issue here seems to be that very little is known about simple skew braces.

\begin{question*}Characterise the natural numbers $n$ for which every skew brace of order $n$ is soluble in the sense of \cite{ballester}.
\end{question*}

\section*{Acknowledgements}

\noindent This research is supported by JSPS KAKENHI Grant Number 24K16891, and has mostly been carried out during the visiting period of the first and second authors at the Department of Mathematics of Ochanomizu University, to which Ferrara and Trombetti would like to express their gratitude for the support and hospitality. Moreover, the first and second authors are members of the non-profit association ‘‘AGTA --- Advances in Group Theory and Applications’’ (www.advgrouptheory.com), and are supported by GNSAGA (INdAM). Funded by the European Union - Next Generation EU, Missione 4 Componente 1 CUP B53D23009410006, PRIN 2022- 2022PSTWLB - Group Theory and Applications.


\bigskip\bigskip


\begin{thebibliography}{99}

\bibitem{pq}
E. Acri and M. Bonatto, \textit{Skew braces of size $pq$}, Comm. Algebra 48 (2020), no. 5, 1872--1881.



\bibitem{numbersolution}
\"O. Akg\"un, M. Mereb, and L. Vendramin, \textit{Enumeration of set-theoretic solutions to the Yang--Baxter equation}, Math. Comp. 91 (2022), no. 335, 1469--1481.

\bibitem{p3}
D. Bachiller, \emph{Classification of braces of order $p^3$}, J. Pure Appl. Algebra 219 (2015), no. 8, 3568--3603.

\bibitem{supersoluble}
A. Ballester-Bolinches, R. Esteban-Romero, M. Ferrara, V. Pérez-Calabuig, and M. Trombetti, \emph{Finite skew braces of square-free order and supersolubility}, Forum Math. Sigma 12 (2024), Paper No. e39, 33 pp.

\bibitem{centralnilp}
\bysame, \emph{Central nilpotency of left skew braces and solutions of the Yang–Baxter equation}, Pacific J. Math. 335 (2025), no. 1, 1--32.

\bibitem{ballester}
A. Ballester-Bolinches, R. Esteban-Romero, P. Jimen\'{e}z-Seral, and V. P\'{e}rez-Calabuig, \emph{Soluble skew left braces and soluble solutions of the Yang-Baxter equation}, Adv. Math. 455 (2024), 109880, 27pp.

\bibitem{Jordan}
A. Ballester-Bolinches, R. Esteban-Romero, V. P\'{e}rez-Calabuig, \emph{A Jordan-H\"{o}lder theorem for skew left braces and their applications to multipermutation solutions of the Yang-Baxter equation}, Proc. Roy. Soc. Edinburgh Sect. A 154 (2024), no. 3, 793--809.

\bibitem{lambdahom}
V.G. Bardakov, M.V. Neschchadim, and M.K. Yadav,
\emph{On $\lambda$-homomorphic skew brace},
J. Pure Appl. Algebra 226 (2022), no. 6, Paper No. 106961, 37 pp.

\bibitem{symmetric}
\bysame, {\it Symmetric skew braces and brace systems}, Forum Math. 35 (2023), no. 3, 713--738.

\bibitem{Baxter}
R.J. Baxter, {\it Partition function of the eight-vertex lattice model}, Ann. Physics 70 (1972), 193--228. 

\bibitem{bonatto}
M. Bonatto and P. Jedli\v{c}ka, \textit{Central nilpotency of skew braces}, J. Algebra Appl. 22 (2023), no. 12, Paper No. 2350255, 16 pp.

\bibitem{Byott}
N.P. Byott, {\it Uniqueness of Hopf Galois structure for separable field extensions}, Comm. Algebra 24 (1996), no. 10, 3217--3228. Corrigendum, {\it ibid.} no. 11, 3705.

\bibitem{soluble}
\bysame, {\it Solubility criteria for Hopf-Galois structures}, New York J. Math. 21 (2015), 883--903. 

\bibitem{Byott transitive}
\bysame, {\it On insoluble transitive subgroups in the holomorph of a finite soluble group}, J. Algebra 638 (2024), 1--31.

\bibitem{Caranti}
A. Caranti, {\it Bi-skew braces and regular subgroups of the holomorph}, J. Algebra 562 (2020), 647--665.

\bibitem{stefanelli}
F. Catino, I. Colazzo, and P. Stefanelli,
\emph{Skew left braces with non-trivial annihilator},
J. Algebra Appl. 18 (2019), no. 2, 1950033, 23 pp.



\bibitem{relation} 
F. Ced\'{o} and J. Okni\'{n}ski, \textit{Indecomposable solutions of the Yang--Baxter equation of square-free cardinality}, Adv. Math. 430 (2023), Paper No. 109221, 26 pp.

\bibitem{nilpotent}
F. Ced\'{o}, A. Smoktunowicz, and L. Vendramin, \textit{Skew left braces of nilpotent type}, Proc. Lond. Math. Soc. (3) 118 (2019), no. 6, 1367--1392.

\bibitem{Childs book}
L.N. Childs, \emph{Taming Wild Extensions: Hopf Algebras and Local Galois Module Theory}, Mathematical Surveys and Monographs, 80. American Mathematical Society, Providence, RI, 2000.

\bibitem{Childs}
\bysame, \emph{Bi-skew braces and Hopf Galois structures}, New York J. Math. 25 (2019), 574--588.

\bibitem{ChildsYBE}
\bysame, {\it Skew left braces and the Yang-Baxter equation}, New York J. Math. 30 (2024), 649--655.

\bibitem{colazzo}
I. Colazzo, M. Ferrara, and M. Trombetti, \emph{On derived-indecomposable solutions of the Yang--Baxter equation}, Publ. Math. 69 (2025), no. 1, 171--193.

\bibitem{thesissuper} 
L. Crew, \textit{On the characterization of the numbers $n$ such that any group of order $n$ has a given property $P$}, Honor Thesis, University of Maryland, College Park, US (2015).

\bibitem{Dickson}
L.E. Dickson, \textit{Definitions of a group and a field by independent postulates}, Trans. Amer. Math. Soc. 6 (1905), no. 2, 198--204. Erratum, \textit{ibid}. no. 4, 547.


\bibitem{GAP}
The GAP Group, GAP – Groups, Algorithms, and Programming, Version 4.12.2; 2022

\bibitem{GP}
C. Greither and B. Pareigis, \emph{Hopf Galois theory for separable field extensions}, J. Algebra 106 (1987), no. 1, 239--258.

\bibitem{GV}
L. Guarnieri and L. Vendramin, \emph{Skew braces and the Yang--Baxter equation}, Math. Comp. 86 (2017), no. 307, 2519--2534.

\bibitem{JVV}
E. Jespers, A. Van Antwerpen, and L. Vendramin, {\it Nilpotency of skew braces and multipermutation solutions of the Yang-Baxter equation}, Commun. Contemp. Math. 25 (2023), no. 9, Paper No. 2250064, 20 pp.

\bibitem{opposite}
A. Koch and P. J. Truman, \emph{Opposite skew left braces and applications}, J. Algebra 546 (2020), 218--235.
 
\bibitem{one-generator}
L.A. Kurdachenko and I.Ya. Subbotin, {\it On the structure of some one-generator braces}, Proc. Edinb. Math. Soc. (2) 67 (2024), no. 2, 566--576. 

\bibitem{two-sided}
T. Nasybullov, {\it Connections between properties of the additive and the multiplicative groups of a two-sided skew brace}, J. Algebra 540 (2019), 156--167.


\bibitem{redei} 
L. Rédei, \emph{Das ‘‘schiefe Produkt’’ in der Gruppentheorie mit Anwendung auf die endlichen nichtkommutativen Gruppen mit lauter kommutativen echten Untergruppen und die Ordnungszahlen, zu denen nur kommutative Gruppen g\"oheren},  Comment. Math. Helv. 20 (1947), 225--264.

\bibitem{Robinson}
D.J.S. Robinson, {\it A course in the theory of groups}. Second edition. Graduate Texts in Mathematics, 80. Springer-Verlag, New York, 1996. 

\bibitem{Rump0}
W. Rump, \emph{Braces, radical rings, and the quantum Yang-Baxter equation}, J. Algebra 307 (2007), no. 1, 153--170.

\bibitem{Rump01}
\bysame, \emph{One-generator braces and indecomposable set-theoretic solutions to the Yang--Baxter equation}, 
Proc. Edinburgh Math. Soc. (2) 63 (2020), no. 3, 676--696.


\bibitem{SV}
A. Smoktunowicz and L. Vendramin, \emph{On skew braces} (with an appendix by N. Byott and L. Vendramin), J. Comb. Algebra 2 (2018), no. 1, 47--86.

\bibitem{ST}
L. Stefanello and S. Trappeniers, \emph{On the connection between Hopf--Galois structures and skew braces}, Bull. Lond. Math. Soc. 55 (2023), no. 4, 1726--1748.

\bibitem{ST2}
\bysame{}, \emph{On bi-skew braces and brace blocks}, J. Pure Appl. Algebra 227 (2023), no. 5, Paper No. 107295, 22 pp.



\bibitem{Trap}
S. Trappeniers, \emph{On two-sided skew brace}, J. Algebra 631 (2023), 267--286.

\bibitem{Tr23} M. Trombetti, \emph{The structure skew brace associated with a finite non-degenerate solution of the Yang-Baxter equation is finitely presented}, Proc. Amer. Math. Soc. 152 (2024), no. 2, 573--583.

\bibitem{leandro}
L. Vendramin, \emph{Skew braces: A brief survey},
in: P. Kielanowski, D/ Beltita, A. Dobrogowska, and T. Goli\'{n}ski, (eds) Geometric Methods in Physics XL. WGMP 2022. Trends in Mathematics. Birkh\"{a}user, Cham.

\bibitem{Yang}
C.N. Yang, {\it Some exact results for the many-body problem in one dimension with repulsive delta-function interaction}, Phys. Rev. Lett. 19 (1967), 1312--1315.

\end{thebibliography}
\end{document}